\documentclass[12pt,a4paper]{amsart}

\usepackage{amsmath}
\usepackage{amssymb,amscd}
\usepackage{enumerate}

\usepackage{fullpage}
\usepackage{hyperref}
\usepackage{ulem}
\theoremstyle{plain}
\newtheorem{theorem}{Theorem}[section]
\newtheorem{lemma}[theorem]{Lemma}
\newtheorem{proposition}[theorem]{Proposition}
\newtheorem{corollary}[theorem]{Corollary}

\theoremstyle{definition}
\newtheorem{definition}[theorem]{Definition}

\newtheorem{question}[theorem]{Question}
\newtheorem{remark}[theorem]{Remark}

\def\Aut{\operatorname{Aut}}

\def\Lie{\operatorname{Lie}}

\def\alg{\operatorname{alg}}

\def\ZZ{{\mathbb Z}}

\def\TT{{\mathbb T}}
\def\UU{{\mathcal U}(X)}

\def\NN{{\mathbb N}}
\def\K{{\mathbb K}}
\def\G{{\mathbb G}}

\def\embed{\hookrightarrow}

\def\U{{\mathcal U}}

\def\SAut{\operatorname{\mathrm SAut}}

\makeatletter
\newcommand\thankssymb[1]{\lowercase{\textsuperscript{\@alph{#1}}}}
\makeatother

\address{
  Kharkevich Institute for Information Transmission Problems\\
  19 Bolshoy Karetny per., 127994 Moscow, Russia
}

\address{
National Research University Higher School of Economics\\
 20 Myasnitskaya ulitsa, Moscow 101000, Russia 
}

\email{a@perep.ru}

\address{\noindent Institut f\"{u}r Mathematik, Friedrich-Schiller-Universit\"{a}t Jena,   Jena 07737, Germany}
\email{andriyregeta@gmail.com}
\usepackage{xcolor}

\DeclareMathOperator{\id}{id}

\begin{document}

\title[Automorphism groups  consisting of algebraic elements]{Automorphism groups of affine varieties consisting of algebraic elements}

\author{Alexander Perepechko\thankssymb{1}}
\thanks{\thankssymb{1}The research of the first author was
carried out at the HSE University at the expense of the Russian Science Foundation (project
no. 21-71-00062)}

\author{Andriy Regeta%\thankssymb{2}
}

\maketitle

\begin{abstract} 
Given an affine algebraic variety $X$, we prove that if the neutral component $\mathrm{Aut}^\circ(X)$ of the automorphism group consists of algebraic elements,
then it is nested, i.e., is a direct limit of algebraic subgroups.
This  improves  our earlier result \cite{PR}.
To prove it, we obtain the following fact.
If a connected ind-group $G$ contains a closed connected nested ind-subgroup $H\subset G$,
and for any $g\in G$ some positive power of $g$ belongs to $H$, then $G=H.$
\end{abstract}

\section{Introduction}
 In this note we work over an
 algebraically closed uncountable field of  characteristic zero  $\K$.  We study the automorphism groups of affine varieties. 
It is well known that these groups can be larger than any algebraic group. For example, the automorphism group $\Aut(\mathbb{A}^n)$ of the affine $n$-space $\mathbb{A}^n$  contains a copy of
a polynomial ring in $n-1$ variables, hence it is infinite-dimensional for $n \ge 2$.

 In \cite{Sh66} Shafarevich introduced the notion of the infinite-dimensional algebraic group, which is currently called 
the \emph{ind-group} and showed that $\Aut(\mathbb{A}^n)$ has the structure of the ind-group. 
Later it was shown that $\Aut(X)$ has a natural structure of an ind-group for any affine variety $X$,  see \cite[Section 5]{FK} and also \cite[Section 2]{KPZ1}.

We call an element $g$ of the automorphism group $\Aut(X)$ \emph{algebraic} if there is an algebraic subgroup $G$ of the ind-group $\Aut(X)$ that contains $g$.  We also denote by  $\mathbb{G}_a$ the additive group of the field and by $\UU\subset \Aut(X)$ the (possibly trivial) subgroup generated by all the $\G_a$-actions. 
It is usually called the \emph{special automorphism group}   and is also denoted by $\SAut(X)$.

Recall that a closed subgroup $G \subset \Aut(X)$ is called \emph{nested} if it is a
countable increasing union of closed algebraic subgroups. It can be shown that this is equivalent to $G$ equal the union of algebraic subgroups.
In \cite{PR} we proved that for  the subgroup
 $\Aut_{\text{alg}}(X) \subset \Aut(X)$ 
generated by all connected algebraic subgroups the following conditions are equivalent:
 \begin{itemize}
\item $\UU$ is abelian;
\item all elements  of $\Aut_{\alg}(X)$ are algebraic;
\item the subgroup $\Aut_{\alg}(X) \subset \Aut(X)$ is a closed nested ind-subgroup; 
\item $\Aut_{\alg}(X) = \TT \ltimes \UU$, where $\TT$ is a maximal subtorus of $\Aut(X)$, and $\UU$ is closed in $\Aut(X)$. 
\end{itemize}

In this paper we prove that this
 result can be partially extended from $\Aut_{\alg}(X)$
 to the connected component $\Aut^\circ(X)$. More precisely, we have the following result which is proved in Section \ref{sectionneutralcomponent}.

\begin{theorem}\label{main}
Let $X$ be an affine variety. The following conditions are equivalent:
\begin{enumerate}[(1)]
\item all elements  of $\Aut^\circ(X)$ are algebraic;
\item the subgroup $\Aut^\circ(X) \subset \Aut(X)$ is a closed nested ind-subgroup;
\item $\Aut^\circ(X) = \TT \ltimes \UU$, where $\TT$ is a maximal subtorus of $\Aut(X)$, and  $\UU$ is abelian and consists of all unipotent elements of $\Aut(X)$.
\end{enumerate}
 \end{theorem}
 
 In \cite{KPZ1} this theorem is proved for algebraic surfaces with a nontrivial group $\UU$.
 
 The key observation in our proof is as follows.
 Under condition (1) for any element of $\Aut^\circ(X)$ some positive power of it belongs to $\TT \ltimes \UU$, see Lemma~\ref{lm:gd}. 
 In Section~\ref{sectionfiniteorderelements} we prove that 
 a connected
 ind-group $G$ coincides with its closed connected nested ind-subgroup $H$ if for any element of $G$ some positive power of it lies in $H$, see Theorem~\ref{th:main-ind}.
 
 In  Section \ref{sectionautomorphismsofrigidvarieties} we also state some observations about the group of automorphisms of a rigid affine variety, i.e., an affine variety that admits no $\mathbb{G}_a$-actions.
 See also \cite{AG} for the case when a rigid affine variety $X$ admits only constant $\TT$-invariant regular functions.

 \subsubsection*{Acknowledgement.} The authors are grateful to Mikhail Zaidenberg and to the referee for useful remarks and suggestions.
\section{Preliminaries}\label{preliminaries}

\subsection{Ind-groups}

The notion of an ind-group goes back to Shafarevich who called these objects infinite dimensional groups (see \cite{Sh66}). 
We refer to \cite{FK}   for basic notions in this context.

\begin{definition}
An ind-variety $V$ is a set together with an ascending filtration
$V_0 \embed V_1 \embed V_2 \embed \ldots \subset V$  such that the following holds:
\begin{enumerate}[(1)]
\item $V = \bigcup_{k \in \NN} V_k$;
\item each $V_k$ is an algebraic variety;
\item for all $k \in \NN$ the embedding $V_k \embed V_{k+1}$ is closed in the Zariski topology.
\end{enumerate}
\end{definition}

An ind-variety $V$ is called \emph{affine} if all $V_i$ are affine.
An ind-variety $V$ has a natural \emph{topology}: a subset $S \subset V$ is called open (resp. closed) if $S_k := S \cap  V_k \subset V_k$ is open (resp. closed) for all $k \in \mathbb{N}$.  A  closed subset $S \subset V$ has a natural structure of an ind-variety and is called an ind-subvariety.

The product of  ind-varieties $X=\bigcup_i X_i$ and $Y=\bigcup_i Y_i$ is defined as $\bigcup_i (X_i\times Y_i)$. 
A \emph{morphism} between ind-varieties $V = \bigcup_k V_k$ and $W = \bigcup_m W_m$ is a map $\phi: V \to W$ such that for every $k \in \mathbb{N}$ there is an $m \in \mathbb{N}$ such that 
$\phi(V_k) \subset W_m$ and that the induced map $V_k \to W_m$ is a morphism of algebraic varieties. 
 This allows us to give the following definition.

\begin{definition}
An ind-variety $G$ is said to be an \emph{ind-group} if the underlying set $G$ is a group such that the map $G \times G \to G$,  $(g,h) \mapsto gh^{-1}$, is a morphism.
\end{definition}

A \emph{closed subgroup}  $H$ of $G$ is a subgroup that is also a closed subset. Then $H$ is again an ind-group with respect to the induced ind-variety structure.
 A closed subgroup $H$ of an ind-group $G = \bigcup_i G_i$ is called an \emph{algebraic subgroup} if $H$ is contained in some $G_i$.  
 
The next result can be found in   \cite[Section 5]{FK}.
\begin{proposition}\label{ind-group}
Let $X$ be an affine variety. Then $\Aut(X)$ has the structure of an affine ind-group such that  a regular action of an algebraic group $G$ on $X$ induces a homomorphism of ind-groups $G \to \Aut(X)$.
\end{proposition}

Two ind-structures  $V=\bigcup_i V_i$ and $V=\bigcup_i V_i^\prime$
are called \emph{equivalent}, if the identity map $\bigcup_i V_i\to \bigcup_i V_i^\prime$ is an isomorphism of ind-varieties.
One also calls $\bigcup_i V_i^\prime$  an \emph{admissible} filtration of the ind-variety $V=\bigcup_i V_i$.

\begin{definition}[{\cite[Definition 1.9.4]{FK}}]
A point $p$ in an ind-variety $V$ is called \emph{geometrically smooth}, if there exists an admissible filtration $V=\bigcup_i V_i$ such that $p$ is a smooth point of $V_i$ for each $i$.
\end{definition}

An element $g \in \Aut(X)$ is called \emph{algebraic} if there is an algebraic subgroup $G \subset \Aut(X)$ such that $g \in G$. 
 An ind-group $G=\bigcup_i G_i$ is called \emph{nested}  if it admits an admissible
filtration $G=\bigcup_i G_i$, where all $G_i$ are algebraic subgroups.
 
\begin{remark}
    If $H$ is a nested ind-group, then all its points are geometrically smooth.
    Indeed, since algebraic groups are smooth, any filtration by algebraic subgroups fits.

However, to our knowledge, this property is not proven for arbitrary ind-groups.
For example, a stronger property of being \emph{strongly smooth} does not hold for the ind-group $\Aut(\mathbb{A}^2)$.
More generally, this group does not admit a filtration by normal varieties, see \cite[Corollary 14.1.2]{FK}.
\end{remark}

\subsection{Lie algebras of ind-groups}\label{Lie}
For an ind-variety $V = \bigcup_{k \in \mathbb{N}} V_k$ we can define the tangent space in $x \in V$ in the obvious way: we  have  $x \in V_k$ for $k \ge k_0$, and
$T_x V_k \subset T_x V_{k+1}$ for $k \ge k_0$, and then we define
$$
  T_x V := \bigcup_{k \ge k_0} T_x V_k,
$$
which is a vector space of at most countable dimension.

For an  ind-group $G$, the tangent space $T_e G$ has a natural structure of a Lie algebra which is  denoted by $\Lie G$, see \cite[Section 4]{Ku} and \cite[Section 2]{FK} for details. 

\subsection{$\G_a$-actions}
 Given an affine variety $X$, we denote by
 $\Aut_{\mathrm{alg}}(X) \subset \Aut(X)$  the subgroup
generated by all connected algebraic subgroups
of the automorphism group $\Aut(X)$. 

An element $u\in\Aut(X)$ is called \emph{unipotent} if $u$ either is an identity or belongs to an algebraic subgroup of  $\Aut(X)$ isomorphic to  $\mathbb{G}_a$. 
 We denote 
 the automorphism subgroup of $\Aut(X)$ generated  by all the unipotent elements by $\UU$.

The groups $\Aut_{\mathrm{alg}(X)}$ and $\UU$ are also denoted by $\mathrm{AAut}(X)$ and $\mathrm{SAut}(X)$ respectively.

\section{When all elements modulo a nested subgroup are of finite order}\label{sectionfiniteorderelements}
The aim of this section is to prove Theorem  \ref{th:main-ind}.

\begin{theorem}\label{th:main-ind}
 Let $G$ be a connected ind-group and $H\subset G$ be a closed connected nested ind-subgroup.
Assume that for any $g\in G$ there exists a positive integer $d$ such that $g^d\in H$. 
Then $G=H.$
\end{theorem}

Since $G$ and $H$ are connected ind-groups, $G$ and $H$ are irreducible curve-connected ind-groups (see \cite[Proposition 2.2.1 and Remark 2.2.3]{FK}). 
By \cite[Proposition 1.6.3]{FK} we can assume that $G$ is filtered by irreducible varieties $G_i$.
Since $H$ is nested, it is filtered by algebraic subgroups $H_i$ of $G$. We may and will assume $H_i$ to be connected, because $H$ is connected.

Consider the multiplication map \[\mu_d\colon G^d = \underbrace{G \times\dots\times G}_\text{$d$ times}\to G, \; (g_1,\ldots,g_d)\mapsto g_1\cdots g_d.
\]
  and its differential 
\[
\mathrm{d}\mu_d\colon (\Lie G)^d = \underbrace{\Lie G \times\dots\times \Lie G}_\text{$d$ times} \to \Lie G.
\]
We have the following statement.
\begin{lemma}\label{lm:lie-d-sum}
Given $(x_1,\ldots,x_d) \in (\Lie G)^d$, the following holds:
\[
\mathrm{d}\mu_d((x_1,\ldots,x_d))= x_1+\cdots+x_d.
\]
\end{lemma}
\begin{proof} 
By linearity, 
\begin{equation}\label{eq:addition}
\mathrm{d}\mu_d((x_1,\ldots,x_d))=
\sum_i \mathrm{d}\mu_d((0,\ldots,0,x_i,0,\ldots,0)).
\end{equation}
We claim that $\mathrm{d}\mu_d((0,\ldots,0,x_i,0,\ldots,0)) = x_i$. Indeed,  let us denote
\[
s_i\colon G \to \underbrace{G \times\dots\times G}_\text{$d$ times}, \; g\mapsto (id,\dots, \underbrace{g}_\text{$i$-th position},\dots,id).
\]
The composition $\mu_d \circ s_i$ is the trivial automorphism of $G$. Hence,
\begin{equation}\label{eq:composition}
\mathrm{d}(\mu_d \circ s_i)\colon
\Lie G \overset{\mathrm{d}s_i}{\to} \underbrace{\Lie G \oplus\dots\oplus \Lie G}_\text{$d$ times} \overset{\mathrm{d}\mu_d}{\to} \Lie G
\end{equation}
is the identity map, where the first map in \eqref{eq:composition} is given by the embedding into the $i$-th coordinate.
Therefore,
we conclude that $\mathrm{d}\mu_d((0,\ldots,0,x_i,0,\ldots,0)) = x_i$. Now, from \eqref{eq:addition} it follows that \[\mathrm{d}\mu_d((x_1,\ldots,x_d))=\sum_i x_i.
\]
\end{proof}

\begin{definition}
We denote $\phi_d\colon G\to G,\ g\mapsto g^d.$
It is an endomorphism of an ind-variety.
\end{definition}

\begin{corollary}\label{cor:x-dx}
The differential
$\mathrm{d}\phi_d\colon \Lie G\to\Lie G$ satisfies
\[\mathrm{d}\phi_d(x)=d\cdot x\]
for any $x\in\Lie G.$
\end{corollary}
\begin{proof}
Consider the embedding 
\[
s\colon G \to \underbrace{G \times\dots\times G}_\text{$d$ times} \; ;\quad g \mapsto (g,\dots,g).
\]
Its differential is the embedding
\[
\mathrm{d}s\colon \Lie G \to \underbrace{\Lie G \oplus\dots\oplus \Lie G}_\text{$d$ times}; \; \; x \mapsto (x,\dots,x).
\]
Since $\phi_d = \mu_d \circ s$, by Lemma~\ref{lm:lie-d-sum}
\begin{equation}\label{eq:x-dx}
\mathrm{d}\phi_d(x)=\mathrm{d}\mu_d((x,\ldots,x))=d\cdot x.%\in\Lie H_k \subset \Lie H.
\end{equation}
\end{proof}

\begin{definition}
For each $d,k\in\NN$ we denote
\[
X_{d,k} = \phi_d^{-1}(H_k)=\{g \in G\mid g^d\in H_k\}\subset G.
\]
\end{definition}

Note that $X_{d,k}$ is not necessarily algebraic.

\begin{definition}
Given a morphism of ind-varieties $\psi\colon X\to Y$ and admissible filtrations $X=\bigcup_i X_i$, $Y=\bigcup_j Y_j$,
we denote by $\eta_\psi\colon \NN\to\NN$ the map such that $Y_{\eta_\psi(i)}$ is the minimal filtration element containing $\psi(X_i)$.
We call $\eta_\psi$ the \emph{index map} of $\psi$. 
We do not specify the filtrations if their choice is clear.
\end{definition}

We will shortly denote $\eta_d=\eta_{\phi_d}$.
In the following lemma we use the assumption that the field $\K$ is uncountable. 

\begin{lemma}\label{lm:Gk-d}
For any $k\in\NN$, there exists $d\in \NN$ such that $\phi_d(G_k)\subset H$.
\end{lemma}
\begin{proof}
Denote 
\[M_d=G_k\cap\phi_d^{-1}(H\cap G_{\eta_d(k)}),\]
which is closed in $G_k$ for each $d$.
Since the irreducible variety $G_k$ is the countable union of the closed subvarieties $M_d$, $d \ge 1$, there exists $M_d$ equal to
$G_k$, see \cite[Lemma 1.3.1]{FK}.
The assertion follows.
\end{proof}

\begin{lemma}\label{lm:A-Xdk}
\begin{enumerate}
    \item The subset $X_{d,k}$ is closed in $G$ for any $d,k\in\ZZ_{>0}$.
    \item\label{X_dkG_i} For any $X_{d,k}$ and $G_i$ there exists $X_{d',k'}$ containing  both $X_{d,k}$ and $G_i$;
    \item\label{X_dkG_i-1} There exists a sequence $\{X_{d_i,k_i}\mid i\in\NN\}$ such that $X_{d_i,k_i}\supset X_{d_{i-1},k_{i-1}}, G_i$.
\end{enumerate}
\end{lemma}
\begin{proof}
The map $\phi_d$ is a morphism of ind-varieties, so
the first statement follows from $X_{d,k}=\phi_d^{-1}(H_k)$.

To prove \eqref{X_dkG_i}, choose a positive integer $c$ such that $\phi_{c}(G_i)\subset H$ by Lemma~\ref{lm:Gk-d} and take $d'=cd$. 
We also take the identity map  $\iota\colon \bigcup_iH_i\to \bigcup_iG\cap H_i$ between filtrations on $H$, as well as the reverse one,  and consider the corresponding index maps  $\eta_\iota$ and $\eta_{\iota^{-1}}$.
The latter exists, since for any $i$ the subset $H\cap G_i$ is exhausted by closed subsets  $H_k\cap G_i$, hence coincides with $H_k\cap G_i$ for some $k$.

Then
\[\phi_{d'}(G_i)=
\phi_{d}\circ\phi_{c}(G_i)\subset
\phi_{d}(H\cap G_{\eta_{c}(i)})\subset 
H\cap G_{\eta_{d}(\eta_{c}(i))}\subset H_{\eta_{\iota^{-1}}(\eta_{d}(\eta_{c}(i)))}.\]
So, if $k'\ge \eta_{\iota^{-1}}(\eta_{d}((\eta_{c}(i)))$, then $G_i\subset X_{d',k'}$.

Analogously, 
\[
\phi_{d'}(X_{d,k})\subset \phi_c(H_k)\subset H_k,
\]
so $X_{d,k}\subset X_{d',k'}$ for $k'\ge k$.    

So, \eqref{X_dkG_i} holds for $d'=cd$ and a sufficiently large $k'$.
Now \eqref{X_dkG_i-1} follows iteratively.
\end{proof}

\begin{proof}[Proof of Theorem \ref{th:main-ind}]
Let \[
\phi_{d,k}\colon X_{d,k} \to H_k, \; \; g \mapsto g^d
\]
be the restriction of
$\phi_d\colon G\to G,\ g\mapsto g^d$.
 Its differential map at the identity,
\[
\mathrm{d}(\phi_{d,k})_{\id}\colon T_{\id}X_{d,k} \to T_{\id}H_k,
\]
is
given by $x \mapsto d\cdot x$ due to Corollary~\ref{cor:x-dx}. This map has trivial kernel and is surjective due to $H_k\subset X_{d,k}$. So,  $\dim T_{\id}X_{d,k} = \dim T_{\id} H_k=\dim H_k$, since $H_k$ is smooth at the identity. 

Let $Y$ be the union of the irreducible components of subsets $X_{d,k}\cap G_i$, $i\in\NN$, containing the identity. 
From $H_k\subset Y$ and $\dim T_{\id}Y=\dim H_k$ we infer that $Y=H_k$. Indeed, otherwise $Y$ contains an irreducible algebraic subset $S$ which contains the identity and is not contained in $H_k$. 
Then $\dim T_{\id}Y\ge\dim T_{\id}(S\cup H_k)>\dim H_k$.
Thus, the set $X_{d,k}$ contains $H_k$ as an irreducible component, and other components do not contain the identity.

By Lemma~\ref{lm:A-Xdk}, for any $i\in\NN$ there exist $d,k\in\NN$ such that $G_i\subset X_{d,k}$.
Since $G_i$ is irreducible and contains the identity, 
$G_i$ is a subset of the only irreducible component of $X_{d,k}$ which contains the identity, namely, $H_k$.
We conclude that $G\subseteq H.$
\end{proof}

\section{Neutral component with only algebraic
elements}\label{sectionneutralcomponent}
In this section we assume that $\Aut^\circ(X)$ consists of algebraic elements.
By \cite[Main Theorem]{PR}, $\UU$ is an abelian unipotent ind-group (which is trivial, one-dimensional, or infinite-dimensional), 
and the subgroup $\Aut_{\alg}(X)$ generated by connected algebraic subgroups equals $\TT\ltimes \UU$, 
where $\TT$ is a maximal algebraic torus.

\begin{lemma}\label{lm:gd}
For any algebraic element $g\in \Aut^\circ(X)$ there exists a positive integer $d$ such that $g^d\in \TT\ltimes \UU$.
\end{lemma}
\begin{proof}
The Zariski closure of $\{g^n\mid n\in\ZZ\}$ is an abelian algebraic group, which we denote by $G$. 
The subgroup $G^\circ$ is of finite index in $G$, so we may denote $d=|G/G^\circ|$ and we have $g^d\in G^\circ$.  
Since $G^\circ$ is a connected algebraic group, $G^\circ\subset \TT\ltimes \UU$.
The claim follows.
\end{proof}

\begin{remark}
By \cite[Theorem 1.1]{Br}, for any algebraic group $G$ there is a finite subgroup $H \subset G$ such that $G = H \cdot G^\circ$.
Thus, any algebraic element of $\Aut(X)$ is a product of an element of $\Aut_{\alg}(X)$ and a finite order one.
\end{remark}

As we have mentioned above, 
$\UU$
is a countable increasing union of closed unipotent algebraic subgroups $\UU_k$ of $\Aut(X)$, where $k \ge 0$.
We set $\UU_k=\UU$ for each $k$ if $\UU$ is itself an algebraic group.

\begin{proof}[Proof of Theorem \ref{main}]
Assume that all elements of $\Aut^\circ(X)$ are algebraic.  By 
\cite[Theorem~1.3]{PR}, $\Aut_{\text{alg}}(X) \subset \Aut(X)$ is a closed subgroup that equals $\TT \ltimes \UU$. 
By Lemma~\ref{lm:gd}, we may apply Theorem~\ref{th:main-ind}
to $G=\Aut^\circ(X)$ and $H=\TT\ltimes\UU$  and conclude that $\Aut^\circ(X) = \TT \ltimes \UU$. 
This proves the implication $(1) \Rightarrow (3)$. 
The implications  $(3) \Rightarrow (2) \Rightarrow (1)$  are obvious. 
\end{proof}

\begin{corollary}
Let $X$ be an affine algebraic variety without $\G_a$-actions such that $\Aut^\circ(X)$ consists of algebraic elements.
Then $\Aut^\circ(X)$ is an algebraic torus of dimension at most $\dim X$.
\end{corollary}
\begin{proof}
Follows from Theorem~\ref{th:main-ind} and \cite[Proposition 4.1(a)]{KPZ1}.
\end{proof}

\section{The automorphism group of a rigid variety}\label{sectionautomorphismsofrigidvarieties}

In this section  we do not assume that $\Aut^\circ(X)$ consists of algebraic elements.
Assume that an affine variety $X$ is rigid, i.e., admits no $\G_a$-actions. Since $\U(X)$ is trivial,  \cite[Theorem 1.3]{PR} shows that $\Aut_{\alg}(X)=\TT$ is an algebraic torus. 

\begin{proposition}
The center of $\Aut^\circ(X)$ contains $\TT$.
\end{proposition}
\begin{proof}
The torus $\TT$ is a normal closed subgroup in $\Aut^\circ(X)$. 
The action of $\Aut^\circ(X)$ on $\TT$ by conjugations induces the natural homomorphism $\Aut^\circ(X)\to\mathrm{GL}(n,\ZZ)$, since $\Aut\TT\cong\mathrm{GL}(n,\ZZ)$, where $\TT$ is seen as an algebraic group.
Since $\mathrm{GL}(n,\ZZ)$ is countable, the image of $\Aut^\circ(X)$ is trivial. The assertion follows.
\end{proof}

\begin{remark}
 Any maximal abstract abelian subgroup $G$  of $\Aut^\circ(X)$ is an at most countable extension of $\TT$.  Indeed,  $G$ coincides with its centralizer $C(G)$ in $\Aut^\circ(X)$, otherwise an element $h\in C(G)\setminus G$ commutes with each element of $G$, a contradiction with the maximality of $G$. Hence, $G \subset \Aut^\circ(X)$ is a closed ind-subgroup (\cite[Lemma 2.4]{LRU}). 
 Further, $G$  contains $\TT$, 
 and by \cite[Theorem B]{CRX} the connected component $G^\circ$ is algebraic. 
 So, $G^\circ=\TT$. 
 
 In particular, the only maximal connected abelian 
  ind-subgroup  of $\Aut^\circ(X)$ is $\TT$. 
\end{remark}

\begin{question}
Given a rigid affine variety $X$, what can we say about the subset of algebraic elements of $\Aut^\circ(X)$?
\end{question}

\end{document}